\documentclass{amsart}
\usepackage{amsmath,amssymb,amsthm,url,graphicx,xcolor,cite,amsrefs}
\usepackage{hyperref} 
    \hypersetup{colorlinks=true, linkcolor=blue, citecolor=magenta, filecolor=magenta, urlcolor=cyan}
\usepackage{mathtools} 
    \mathtoolsset{showonlyrefs,showmanualtags}
\allowdisplaybreaks
    \makeatletter
    \def\l@subsection{\@tocline{2}{0pt}{2.5pc}{5pc}{}}
    \def\l@subsubsection{\@tocline{2}{0pt}{5pc}{7.5pc}{}}
    \makeatother
\numberwithin{equation}{section}
\newtheorem{thm}[equation]{Theorem}
\newtheorem{lemma}[equation]{Lemma}
\newtheorem{prop}[equation]{Proposition}

\theoremstyle{definition}
\newtheorem{definition}[equation]{Definition}
\theoremstyle{remark}
\newtheorem{remark}[equation]{Remark}
\DeclareMathOperator{\lcm}{lcm}
\newcommand{\Z}{\mathbb Z}

\newcommand{\A}{\mathbb A}

\newcommand{\T}{\mathbb T}
\newcommand{\F}{\mathcal F}

\begin{document}

\title[Inequalities for Primes in Progressions] {Improving and Maximal Inequalities for Primes in Progressions}

\author[Giannitsi]{Christina Giannitsi}
    \address{School of Mathematics, Georgia Institute of Technology, Atlanta GA 30332, USA}
    \email{cgiannitsi3@math.gatech.edu}
\author[Lacey]{Michael T. Lacey}   
    \address{ School of Mathematics, Georgia Institute of Technology, Atlanta GA 30332, USA}
    \email {lacey@math.gatech.edu}
    \thanks{MTL: The author is a 2020 Simons Fellow. \\ Research of all authors is supported in part by grant  from the US National Science Foundation, DMS-1949206}
\author[Mousavi]{Hamed Mousavi}
    \address{School of Mathematics, Georgia Institute of Technology, Atlanta GA 30332, USA}
    \email{hmousavi6@gatech.edu}
\author[Rahimi]{Yaghoub Rahimi}
    \address{School of Mathematics, Georgia Institute of Technology, Atlanta GA 30332, USA}
    \email{yaghoub.rahimi@gatech.edu}

\begin{abstract}
Assume that $ y < N$  are integers, and that $ (b,y) =1$.
Define an average along the primes in a progression of spacing $ y$,  given by integer  $ (b,y)=1 $.
\begin{align*}
A_{N,y,b} := \frac{\phi (y)}{N} \sum _{\substack{n <N\\n\equiv b\mod{y}}} \Lambda (n) f(x-n)
\end{align*}
Above, $\Lambda $ is the von Mangoldt function and $\phi $ is the totient function.
We establish improving and maximal inequalities for these averages.  These bounds are uniform in the choice of progression.
For instance,
 for $ 1< r <  \infty $  there is an integer $N _{y, r}$ so that for $(b,y)=1$, we have 
\begin{align*}
\lVert \sup _{N>N _{y,r}} \lvert A_{N,y,b} f \rvert \rVert_{r}\ll
\lVert  f\rVert_{r}.
\end{align*}
The implied constant is only a function of $ r$.
The uniformity over progressions imposes several novel elements on the proof.
\end{abstract}

\maketitle
\tableofcontents

\section{Introduction}

We study averages over primes in arithmetic progressions, establishing inequalities with constants  independent of the choice of progression.
As far as we know, these are new.  And the underlying proof entails some new complications, as compared to known results and their proofs.
The averages we are concerned with are defined as follows.
For $b, y \in \mathbb N $ and $y \leq N$,  and function $ f \;:\; \mathbb Z \to \mathbb R $, define
\begin{equation*}
	A_{N,y,b} f := \frac{ \phi (y)}{N} \sum _{\substack{n <N\\n\equiv b\mod{y}}} \Lambda (n) f(x-n),
\end{equation*}
where $\Lambda$ is the von Mangoldt function and $ \phi $ is the Euler totient function.
This is the average of $ f$ along the primes in the arithmetic progression  $ \{  n \;:\; n \equiv b \mod y\}$. We are only interested in the case of $ (b,y)=1$ of course, hence  we use the totient function $ \phi (y)$ above.

 As our first result, we establish $ \ell ^{r}$ improving type inequalities.

\begin{thm}\label{fixedscalelpimproving}
	For $r \in (1,2)$, there exists $C_r>0$, so that for all integers $ y$, there is a $N _{r,y}>0$ such that for all $N>N _{r,y}$ and compactly supported function $f$,
	\begin{align} \label{e:fixedscalelrimrroving}
 \max _{ (b,y)=1} \lVert  A_{N,y,b}f\rVert _{\ell ^{r'}} \leq C_r \left(\frac{y}{N}\right)^{\frac{1}{r}-\frac{1}{r'}} \,  \lVert f\rVert _{\ell ^{r}}.
	\end{align}
  Above, we set $ N _{r,y} = e ^{C_{r, \delta  } y ^{ \delta  }}$,
 for any $ \delta >0$, and $C _{r, \delta }$ sufficiently large.
\end{thm}

The right hand side of \eqref{e:fixedscalelrimrroving} is the correct scale factor for the inequality to hold uniformly in $ N$.
And, it is sharpest when $f$ is assumed to be supported on a progression of spacing $y$.
It is natural to suppose that $N$ is sufficiently large, as a function of $y$.
For the average over all primes, this inequality was established in \cite{laceyprimes}, with study of the endpoint case in \cite{lacey2021endpoint}.
The novelty here is the uniformity in choice of arithmetic progression.

We also study the maximal inequality.

\begin{thm}\label{t:maximal}  For $ 1 < r < \infty $, there is a constant $ C_r$ so that for all  integers $ y$,  there is a $N _{r,y}>0$ so that
\begin{equation}\label{e:maximal}
 \lVert  \sup _{N> N _{r,y}} \lvert  A _{N,y,b} f \rvert  \rVert _{\ell ^r}
\leq C_r   \lVert f\rVert _{\ell^r}
\end{equation}
	The inequality above is uniform in $ y$ and $ (b,y)=1$.
\end{thm}

We prove these theorems using the Siegel-Walfisz Theorem, and methods that are common to the study of these averages
and their $ \ell ^{r}$ improving and sparse bounds.
The bounds from the  Siegel-Walfisz Theorem are ineffective. 
So, our bounds are also  ineffective.

The uniformity over the progressions introduces important differences with
prior papers studying averages over the primes. We describe them here.

The Hardy-Littlewood circle method is key.  The decomposition of the
Fourier transform of the averages
leads to two competing sets of properties.
The first, is the \emph{height} of rational points in the circle.
This property was identified by Bourgain \cite{MR1019960}, and
refined by Ionescu and Wainger \cite{MR2188130}.
Its role is well understood.

This height is, for our purposes, dictated by the
size of Gauss sums associated to the rational.
Most commonly, this height is  given by the denominator of the rational
point in its lowest terms.
In our setting, \emph{these are decoupled}.  Rational points whose denominator divides $ y$ all have Gauss sums of magnitude one.
Specializing the discussion to the primes, the Gauss sum associated with rational $a/q$ in lowest terms, is $ \mu (q)/ \phi (q)$.
In our setting, the Gauss sums are given by a Ramanujan type sum along a progression.
These are  evaluated in Lemma~\ref{l:ZERO}.
And, the  height of $ a/q$ is given by $ \operatorname {lcm} (q,y)/y$.  In particular, there are more than $y$ rational points of height one.

This is a novel feature, and once identified, only adds a little extra difficulty to the
proof of the improving inequality.
In particular, the formulation of the Fourier multiplier approximation theorem, Theorem \ref{t:appthm},
 is different from standard statements of this type.
 For the maximal inequality, however,
one cannot use the standard approach.
The latter approach uses the
Bourgain Multifrequency Maximal Inequality  \cite{MR1019960}. It has a bound that is logarithmic in the number of rationals of a given height.
And so, we cannot appeal to it.
  We use a different inequality at this point. See Lemma \ref{l:notBourgain}.
 Also note that the large number of points of height one would complicate applications of the Ionescu-Wainger theory,
 in seeking $\ell^p$ estimates. But we do not need to confront them, due to our approach to the improving inequalities.

The improving inequalities require a second property, call it a \emph{Ramanujan height}.
It depends upon subtle cancellation and size conditions on certain Ramanujan's sums. Again, there is a complication
in evaluating these sums, and we need a progression variant of a familiar identity due to Cohen,
see Lemma~\ref{l:cohenprogression}.   Applying this identity is not so straightforward.
An inverse Fourier transform calculation, easy in the case of the full sequence of primes, becomes much more involved. See Lemma \ref{inversefourier}.

In addition, one needs to know that Ramanujan's sums are typically of size one.  This is quantified in a famous inequality due to Bourgain, stated in Lemma \ref{l:BourgainRamanujan}.  Again, we need a progression
version, stated in Lemma \ref{l:BourgainRamanujanProgression}.

\smallskip
Bourgain \cites{MR937582,MR937581} initiated the study of these discrete averages, with the $\ell^2$ result for the
square integers being an important breakthrough. The first example of an arithmetic sequence
for which the full $\ell^p$ inequalities were known is Wierdl's result for  the primes \cite{MR995574}.
See Mirek and Trojan \cite{MR3370012} for a discussion of this proof.
Averages along the primes, and closely related objects, have been studied by many,
including variational results by \cite{MR4029173}, thin subsets of the primes \cite{MR3299842},
Carleson type theorems \cite{MR3830238}, and endpoint type results \cites{MR4029173,lacey2021endpoint}.
This is the only paper we are aware of that discusses the uniformity over progressions.

\smallskip

The remainder of the paper begins with \S \ref{s:Prelim}, where some notation
and standard facts are collected.  This section also has the crucial progression
variants of some standard facts about Ramanujan's sums. These facts are probably
known, but we could not find relevant sources to cite, so we include complete proofs
for these facts.  The remaining sections develop the tools along standard lines,
while addressing the complications from the decoupling of the size of the
Gauss sum at rational $a/q$ and $q$ mentioned above.
The circle method is used to build approximation to the multipliers in \S \ref{s:Approx}.
There are differences in the standard approaches here, accounting for the
fact that the different role that height plays in this argument. See Definition \ref{d:height}.
The following  section \S \ref{s:HighLow} develops the
properties of the High and Low decomposition of the multipliers. These definitions
are not completely standard.
The analysis of the Low part depends very much on the progression versions
of the Ramanujan multipliers. The Bourgain Multifrequency Maximal Inequality
cannot be used for the High part.  The concluding section \S \ref{s:proof} is
standard in nature.

\section{Preliminaries} \label{s:Prelim}

For quantities $a$ and $b$, we write $a \ll b$ if $ \lvert  a\rvert  \leq C \, b$ for some constant $C>0$. We  write $a \ll _p b$ if they implied constant depends on $p$.

For a function $f$ on the integers,   $\widehat f$ or $\F f$ denotes the discrete time Fourier transform of $f$, defined as
\begin{equation*}
\F f (\theta) = \sum _{x \in \Z} f(k) e^{-2 \pi i x \theta} ,
\end{equation*}
and $\check f$ or $\F ^{-1}$ the inverse discrete time Fourier transform,
\begin{equation*}
\F^{-1}  f (x) = \int_0^1 \widehat f (\theta) e^{2 \pi i x \theta} \, d \theta.
\end{equation*}
Finally, let $e(x) := e^{2 \pi i x}$.

Let $\Psi$ denote the Chebyshev function,  which counts the primes in a progression.
\begin{equation*}
\Psi (N,y,b) = \sum _{\substack{n<N\\ n\equiv b\mod{y}}} \Lambda (N).
\end{equation*}
The fundamental estimate on it is given here, requiring  that the average be sufficiently large, depending upon $y$ and the $ \ell ^r$ index of the inequalities.

\begin{thm}\label{t:SW}[Siegel-Walfisz Theorem]
Let $ J>1$ be an integer. This holds for all $ x>1$,    $ y \leq (\log N) ^{J}$ and $ b \mod q$.
\begin{equation}\label{e:SW}
\Bigl\lvert
\Psi (x,y,b) - \frac{x} {\phi (y)}
\Bigr\rvert \leq C_J \cdot  x \operatorname {exp}(-c_J  \sqrt{\log x}    ) ,
\end{equation}
where the constants $ C_J$ and $ c_J$ depend only on $ J$.
\end{thm}

Throughout, we denote $ \mathbb A  _q = \{ a \in \mathbb Z /q \mathbb Z  \;:\; (a,q)=1\}$, so that $ \lvert  \mathbb A  _q\rvert = \phi (q) $, the totient function.    This lower bound on the totient function is well known.  For all $ 0< \epsilon <1$, we have
\begin{equation}\label{e:totient}
\phi (q) \gg q ^{1- \epsilon }.
\end{equation}
We also make use of the major and minor arc decomposition. For integers $q, \, s \geq 1$ consider the following sets
\begin{equation*}
\mathcal R _s = \left\lbrace \frac{a}{q} \in [0,1) \, : a \in \A_q, \;  \;  2^{s-1} \leq  q < 2^{s} \right\rbrace
\end{equation*}
For $0<\varepsilon \leq 1/4$ and  $\frac{a}{q} \in \mathcal R _s$, with $s\leq j \varepsilon$,  we define the j-th major arc at $a/q$ as
\begin{align*}
	\mathfrak{M} _j \left( a/q \right)  
	    & : = \left (\frac{a}{q} - 2^{(\varepsilon -2) j}, \frac{a}{q} + 2^{(\varepsilon -2) j} \right),
\end{align*}
which are disjoint for $\varepsilon$ small enough.
The $j$-th major arcs are given by 
$\mathfrak{M} _j := \bigcup_{\frac{a}{q} \in \mathcal R _s} \mathfrak{M} _j (\frac{a}{q}) $. 
We define the $j$-th minor arcs $\mathfrak{m}_j$ as the complement of $\mathfrak{M}_j$.

We turn to exponential and Ramanujan's sums.
  Define  Ramanujan's sums by
\begin{equation} \label{e:tau}
\tau_q (x)  =  \sum _{a \in \A_q}  e (ax/q).
\end{equation}
Cancellative properties of the Ramanujan's sums are very important for us, and expressed in different ways.
The first of these is
\begin{equation}\label{e:mobius}
\tau _{q} (x)  =  \mu (q) \qquad  (q,x)=1.
\end{equation}
Above, $ \mu  $ is the M\"obius function, the  multiplicative function with $ \mu (p) =-1$ for all primes $p$, that vanishes on integers that are not square free.
A second example of the cancellative properties is
\begin{align} \label{e:smooth}
\sum_{d|r} \tau_d(x) = \begin{cases}
r & \textup{ if } r|x\\
0 & \textup{otherwise}.
\end{cases}
\end{align}
We mention the next cancellation property known as Cohen's identity
\begin{equation}\label{eq:Cohen}
\sum_{\substack{r<q\\  (r,q)=1}} \tau_q(x+r) = \mu(q) \tau_q(-x).
\end{equation}
Their relationship to the prime numbers are well known.
In this study, we will need
 these properties,  as well as certain progression versions of them.

Firstly, we examine  Ramanujan's sum restricted to a progression. This formula must be known, but we were not able to find it in the literature.

\begin{lemma}\label{l:ZERO}
Let $ q, y, b \in \mathbb N $, with $ g = \gcd(q,y)$, $( a,q)=1$,  $(b,g) = 1$.  If $ (g,q/g)=1$, let $1 - g\bar{g}=\frac{q}{g}t $, where $  \overline {g} $ is the multiplicative inverse of $ g$ mod $ q/g$.
Then,
\begin{align} \label{e:ZERO}
 \sum _{\substack{r\in \mathbb A _q\\ r\equiv b\mod{g}}} 
\operatorname e \big( \frac{ra}{q} \big)   =
\begin{cases}
    0   &   \textup{if } 1<g<q \textup{ and }  \big( g, \frac{q}{g} \big) >1
    \\
    \mu \big( \frac{q}{g} \big) \operatorname e \big( \frac{abt}{g} \big)   &   \textup{if } 1\leq g<q \textup{ and }  \big( g, \frac{q}{g} \big) =1
    \\
    e\big( \frac{ab}{q} \big) & \textup{if } g=q .
\end{cases}
\end{align}
\end{lemma}

\begin{proof}
The case of $g=1,q$ are elementary, and we leave them to the reader. Below, we will assume that $g$ is a proper divisor of $q$.  
Let $ u$ be a  divisor of $ q$. We have
\begin{equation} \label{e:uprogression}
\sum_{j=0} ^{\frac{q}{u} -1} \operatorname e ( a(b+ju)/q)
 =
 \begin{cases}
0    &  u < q
\\
\operatorname e (ab/q)
& u=q
 \end{cases}
\end{equation}
In the case of $ u < q$, note that  since $ a\in \mathbb A _q$, we also have $ a\in \mathbb A _{q/u}$, hence $ j\to aj $ is a permutation on $ \mathbb Z/ (q/u) \mathbb Z $.  And, if $ u=q$, there is only a single term in the summation, so there is no cancellation.

For a set $ B\subset \mathbb Z _q$, set
\begin{equation} \label{e:SB}
S (B) = \sum_{ s \in B} \operatorname e (sa/q).
\end{equation}
We need to evaluate the term $S (A)$, where
 $ A  = \{ r \in \mathbb A_q \;:\;   r \equiv b\mod g\}$.
To do so, we use the Inclusion-Exclusion Principle to write $ S (A)$  as a sum of progressions, as in \eqref{e:uprogression}.

Consider the set
$ T_{u}  =  \{  b+j u \;:\; 0\leq j < q/u\} $, and note that \eqref{e:uprogression} is essentially an estimate of $S(T_g)$.
Now, suppose $ g$ is a proper divisor of $ q$. Then, for all prime factors $ p$ of $ g$, we have
\begin{equation*}
b+jg \equiv b \not\equiv 0 \mod p ,
\end{equation*}
since $ (b, g) = 1$.   That is, if $r\in  T_g  \setminus A$, \emph{it must be divided by a
prime factor of $ q$ that does not divide $ g$.}

Let $ U_g$ be all square free proper divisors of $ q $ that are relatively prime to $ g$.
If $ U_g= \emptyset $, that means that $ g$ and $ q$ are powers of the same prime $ p$.  Therefore, $ A = T_g$, since   for every $p|q$, we conclude that $p|g$, and
$ b+ j g \equiv b \not\equiv 0 \mod p$.  
So $S(A)=S(T_g)$ and our desired estimate follows from \eqref{e:uprogression}.

On the other hand, if  $ U_g \neq \emptyset $, consider $ u\in U _{g}$, and let $ r$ denote an integer $ r = b + jg \in T_g$.
We have
\begin{equation*}
b+jg \equiv 0 \mod u  \quad \textup{iff} \quad   j \equiv -b \overline {g} \mod u.
\end{equation*}
This holds since $ (g,u) =1$.   Set $ \beta _u = -b  \overline {g} \mod u$,  (we may have $ b \equiv 0 \mod u$ for some $ u\in U _{g}$,) and $ \beta _g = b$.
Let
\begin{equation*}
R_u  = \{ rg+b \in  \mathbb Z _q \;:\;  r\equiv \beta _u \mod u \},
\end{equation*}
and notice that we can then write $A$ as
\begin{equation*}
A = T_g  \setminus  \bigcup _{u\in U _{g}} R_u.
\end{equation*}

We can now utilize the Inclusion-Exclusion Principle.
Let $ \omega (n)$ be the number of distinct prime factors of $ n$.  Then
\begin{align*}
S (A)  & =S (T_g)+ \sum _{u\in U} (-1) ^{\omega (u)} S (R_u).
\end{align*}
The equation above implies that the desired sum in \eqref{e:ZERO} is a linear combination of other sums that can be expressed in the form of \eqref{e:uprogression}, and can be therefore estimated.
Additionally, \eqref{e:uprogression} forces a lot of the sums above to be zero. Specifically, all of them are zero except for when $ gu=q$. In that case, the progression  consists of a single term.
This forces $ q/g$ to be square free, since $ U_g$ consists of square free integers.
The corresponding coefficient from the Inclusion-Exclusion Principle is $ (-1) ^{\omega (q/g)} = \mu (q/g)$, which means that 
\begin{align*}
S (A)  & = \mu(q/g) S (T_{q/g})
\end{align*}
Recalling the definition of $\beta_{q/g}$, we see that $\beta_{q/g} = -b\bar{g} =\frac{-b}{g} (\frac{q}{g}t-1)$. So $r= b+g\bar{g}= b+\frac{q}{g}t-b=\frac{q}{g}t $. The result follows from   \eqref{e:uprogression}.
\end{proof}

Secondly, we need a progression version of Cohen's identity \eqref{eq:Cohen}.

\begin{lemma}\label{l:cohenprogression}
We have for $g=\gcd(y,q)$
\begin{align}\label{e:cohenprogression}
\sum_{\substack{t\in A_q\\ t\equiv b\mod{g}}} \tau_q(x+t)=
\begin{cases}
  0    &  (g,\frac{q}{g})>1
  \\
\mu (q/g) \tau_{q/g}(x)\tau_g(x+b) &  (g,\frac{q}{g})=1
\end{cases}
\end{align}
\end{lemma}

\begin{remark}
Note that if $g=1$, Lemma \ref{l:cohenprogression} reduces to the usual Cohen's identity. It is expected, because the progression on $y$ and on $q$ become independent. Also if $g=q$, we will get only the term $t=b$ from the sum in the left hand side of \eqref{e:cohenprogression}.  This term is equal $\tau_q(x+b)$, which happens to be the right hand side.
\end{remark}

\begin{proof}
The sum in question is
\begin{align*}
\sum_{\substack{t\in \A_q\\ t\equiv b\mod{g}}} \sum_{r\in \A_q} e\left(\frac{(x+t)r}{q}\right) =  \sum_{r\in \A_q} e\left(\frac{rx}{q}\right) \sum_{\substack{t\in \A_q\\ t\equiv b\mod{g}}} e\left(\frac{tr}{q}\right)
\end{align*}
By Lemma \ref{l:ZERO}, the inner-most sum on the right hand-side is zero, when $\gcd(g,\frac{q}{g})>1$.
Continuing with the assumption that $ \gcd(g,\frac{q}{g})=1$, the sum above is equal to
\begin{align*}
\sum_{r\in A_q}  \operatorname e \left( \frac{r(x + bsq/g)}{ q} \right)   = \mu (q/g) \tau_q(x + bsq/g)
\end{align*}
where $1-g\bar{g} =\frac{q}{g}s$ if $g < q  $ and $s = 1 $ if $g = q$.  Ramanujan's sums are  multiplicative, leading to
\begin{align*}
\sum_{\substack{t\in A_q\\ t\equiv b\mod{g}}} \tau_q(x+t)&= \mu (q/g)  \tau_{q/g}(x + bsq/g)\tau_g(x+ b(1-g\bar{g})) \nonumber\\
&= \mu (q/g)  \tau_{q/g}(x)\tau_g(x+b).
\end{align*}
The last equality follows from the periodicity of  Ramanujan's sum.
\end{proof}

 A final property of Ramanujan's sums is a  fundamental inequality due to Bourgain \cite{bourgain93}.
 It implies that typical values of $ \tau _q (n)$ are approximately $1$, on average.
 \begin{lemma}\label{l:BourgainRamanujan}  Given integer $ k$ and $ \epsilon >0$, we have for all integers  $ M > y Q^{k}$
 \begin{equation} \label{e:BourgainRamanujan}
 \Biggl[  \frac{1}M  \sum_{ \substack{n\leq M}} \Bigl[  \sum_{ \substack{q \leq Q }} \lvert   \tau _q (n)\rvert  \Bigr]  ^{k}\Biggr] ^{1/k} \ll  Q ^{1+ \epsilon }
 \end{equation}
 The implied constant depends only on $ \epsilon $.
 \end{lemma}

  We need a progression version of this inequality.

  \begin{lemma}\label{l:BourgainRamanujanProgression}  Given integer $ t$ and $ \epsilon >0$, and integers $b, y$, with $(b,y)=1$ we have for all integers  $ M > y Q^{t}$
  \begin{equation} \label{e:BourgainRamanujanProgression}
  \Biggl[  \frac{y}M  \sum_{ \substack{n\leq M\\  n \equiv b \mod y }} \Bigl[  \sum_{ \substack{q \leq Q\\   (q,y) =1 }} \lvert   \tau _q (n)\rvert  \Bigr]  ^{t}\Biggr] ^{1/t} \ll  Q ^{1+ \epsilon }
  \end{equation}
  The implied constant depends only on $ \epsilon $.
  \end{lemma}

  Notice that the length of the average is required to grow with $ t$.
  That the constant is independent of $ t$ is not recorded as such in the literature, but follows from a modification of the proof
  in \cite{kesslerlaceymena19lacunary}.  The $ \epsilon $ dependence is traced to an inequality for the divisor function.

  \begin{proof}
  We   follow  the proof from \cite{kesslerlaceymena19lacunary}*{\S3}.
  Firstly, we have $ \lvert \tau _q (n) \rvert \leq (q,n) $.
  Secondly, for $\vec q \in [1 , Q] ^t$, let $\mathcal L (\vec q) $ be the least common multiple of $q_1,\ldots, q_t$.
  We assume throughout that all $q_j$ are relatively prime to $y$.
  The map $m \mapsto \prod _{ j=1}^t \tau _{q_j} (m y +b)$ is periodic with period $\mathcal L (\vec q)$.
  The condition $M> y Q ^{t}$ then implies that for any $\vec q \in [1 , Q] ^t$,
  \begin{equation}
   \frac{y}M  \sum_{ \substack{n\leq M\\  n \equiv b \mod y }}  \prod _{ j=1}^t  (q_j, n)
   \ll
   \frac{1} {\mathcal L (\vec q)}    \sum_{ \substack{n\leq   \mathcal L (\vec q) }} \prod _{ j=1}^t  (q_j, ny+b) .
  \end{equation}
  On the right, we have dropped the modularity assumption on $n$.

  Thirdly, we have, uniformly in $\vec q \in [1,Q]^t$, subject to the condition that $q_j$ are coprime to $y$,
  \begin{equation} \label{eq:LQ}
   \sum_{ \substack{n\leq   \mathcal L (\vec q) }}  \prod _{ j=1} ^t  (q_j, ny+b)  \ll   Q ^{t+ \epsilon } .
  \end{equation}
  We establish this here.  Due to the multiplicative structure of the estimate above, it suffices to consider this case.
  Consider the inequality below for prime $p \nmid y$, and integers $t$ and $x_1\geq x_2\geq \cdots \geq x_t$.
  \begin{equation}\label{eq:px}
  \sum _{ \substack { n\leq p ^{x_1}}} \prod_{j=1}^t(p ^{x_j} , ny+b)   \ll   p ^{x_1t+\epsilon  }.
  \end{equation}
 To see this, note that
\begin{align*}
  \sum _{ \substack { n\leq p ^{x_1}}} \prod_{j=1}^t(p ^{x_j} , ny+b)  \leq  \prod_{j=1}^t p^{x_1-x_j}  \sum _{ \substack { n\leq p ^{x_j}}}(p ^{x_j} , ny+b)
\end{align*}
 For $w_j\leq x_j$,  if $p^{w_j} \Vert  n_i y +b$ for $i=1,2$ and $n_1\neq n_2 \leq p^{x_j}$. Then,  $p^ {w_j}\vert  (n_1-n_2)$,
  since $p\nmid y$.
  That is, there are at most $p ^{x_j-w_j}$ values of $n$ such that $p^{w_j} \Vert ny+b$.
 It follows that
 \begin{align*}
   \sum _{ \substack { n\leq p ^{x_1}}}\prod_{j=1}^t (p ^{x_j} , ny+b)    &
   \leq \prod_{j=1}^t p^{x_1-x_j}   \sum _{w_j\leq x_j }  p ^{x_j -w_j} p ^{w_j} \nonumber\\
 & \ll p^{x_1t} \prod_{j=1}^t   \sum _{w_j\leq x_j }  1 \ll  p ^{x_1t} \prod_{j=1}^t x_j \ll  p^{x_1t+\epsilon}.
 \end{align*}

  Fourth, we have the bound
  \begin{equation}
  \sum _{\vec q \in [1,Q] ^t} \frac{1} { \mathcal L (\vec q)} \ll  Q ^{ \epsilon }.
  \end{equation}

  Pulling together the different estimates gives us this chain of inequalities, which completes the proof.
  \begin{align*}
  \frac{y}M  \sum_{ \substack{n\leq M\\  n \equiv b \mod y }}  \Bigl[  \sum_{ \substack{q \leq Q\\   (q,y) =1 }} \lvert   \tau _q (n)\rvert  \Bigr]  ^ {t}
  &\ll
    \frac{y}M   \sum _{ \substack { \vec q \in [1,Q] ^ {t} \\  (y, \mathcal L (\vec q))=1}  }
     \sum_{ \substack{n\leq M\\  n \equiv b \mod y }}  \prod _{ j=1} ^t  (q_j, yn+b)
  \\
    &\ll
    \sum _{ \substack { \vec q \in [1,Q] ^ {t} \\  (y, \mathcal L (\vec q))=1}  }
    \frac{1} {\mathcal L (\vec q)}    \sum_{ \substack{n\leq   \mathcal L (\vec q) }} \prod _{ j=1} ^t  (q_j, yn+b)
  \\
     &\ll
    \sum _{ \substack { \vec q \in [1,Q] ^ {t} }  }  \frac{Q ^ {t+ \epsilon }}{ \mathcal L (\vec q)}
    \ll Q ^ {t + 2 \epsilon  }.
  \end{align*}
  \end{proof}

\section{Approximation} \label{s:Approx}

Our strategy of proving the desired results consists of firstly approximating our kernel by another multiplier. We opt to do that on the Fourier side, and obtain an error that is easily controlled. This is established in Theorem \ref{t:appthm}. The next step is to take a closer look at the approximating multiplier and split it into two pieces, one that is well behaved on the time domain, and one that is well-behaved in the frequency domain. We call these pieces the  Low and High parts and they are thoroughly discussed in the next section. 
The principal result of this section is to prove Theorem \ref{t:appthm}, the approximation result for
\begin{equation*}
\widehat A _{N, y, b} ( \theta )   = \frac {\phi(y)}{N}  \sum_{ \substack{n < N\\ n \equiv b \mod y }} \Lambda (n) e (-n \theta ) .
\end{equation*}
This is the  Fourier transform of our averaging kernel.
The standard average over the integers from $ 1$ to $ N$ is a multiplier with kernel
\begin{equation*}
\widehat M_N (\theta )  = \frac {1}{N} \sum_{n < N} e (-n \theta ).
\end{equation*}
The progression version of the average over the integers congruent to $ b \mod y$, and less than $ N$
is denoted by $M_{N,y,b}$.  As a Fourier multiplier, its kernel is
\begin{align}\label{e:M}
\widehat M _ {N,y,b} ( \theta) &= \frac{y}{   N} \sum _{\substack{n\leq N \\ n \equiv b \mod y}}e(-n\theta)
\\
& = e (-b \theta) \frac{y}{   N} \sum_{ n\leq \frac{N-b}{y}}  e (-ny \theta) 
\\ & = e (-b \theta) \widehat M_{\frac{N-b}{y}} ( y\theta ).
\end{align}
We record an elementary relation between these two definitions.
\begin{align} \label{e:MM}
\widehat M _ {N,y,b} ( \theta)  &= \widehat  M _{\frac{N-b}{y}} (y \theta ) (1+ O(b \lvert  \theta \rvert ))
\end{align}
Also note that because of the relative sizes of $b$ and $y$, we always have $\frac{b}{y}<1$. This means that there can only be at most one integer $n_0 \in [\frac{N-b}{y}, \frac{N}{y})$. Therefore
\begin{align} \label{e:MM'}
    \widehat M _{\frac{N-b}{y}} (\theta) = \begin{cases}
    \frac{N}{N-b} \widehat M _{N/y} (\theta) , & \text{ if } [\frac{N-b}{y}, \frac{N}{y}) \cap \Z = \emptyset 
    \\
    \frac{N}{N-b} \widehat M _{N/y} (\theta)  - \frac{ye(-n_0\theta)}{N-b} , & \text{ if } [\frac{N-b}{y}, \frac{N}{y}) \cap \Z = \{ n_0\}  
    \end{cases}
\end{align}

Let $\|x\|$ denote the distance of real number $x$ from its nearest integer.  For the complete average, the estimate below is elementary.
\begin{equation}\label{knownM}
	\widehat{M}_N(\theta) =  \frac{1}{ N} \sum _{\substack{n\leq N }}    e(-n\theta)  \ll \min \left( 1, \frac{1}{N \lVert  \theta \rVert} \right)
\end{equation}
The progression version of this inequality is
\begin{equation}\label{e:MNless}
\widehat M _ {N,y,b} ( \theta)  \ll  \min \left\{1, \frac{y}{N \lVert y\theta\rVert} \right\} .
\end{equation}

Our primary focus is on the multiplier $ \widehat A _{N,y,b} $. The first step in approximating it is taken here, where we focus our attention around the origin.

\begin{lemma}\label{l:nearZero}
For  all $ J >1$, there is a $ 0< c < 1$ so that for  $|\theta  |<\frac{\log^{J}(N)}{N}$,  and $  y<\log^{J}(N) $, there holds for all $ b  \in \mathbb A _y$
\begin{align}\label{e:NearZero}
 \widehat A _{N,y,b} (\theta )  -  \widehat  M_{N/y} (y\theta )      \ll   \operatorname {exp} (-c \sqrt{\log N}    ).
\end{align}
\end{lemma}

\begin{proof}
We establish the closely related inequality
\begin{equation}\label{e:NNear}
 \widehat A _{N,y,b} (\theta )  -   \widehat  M_{N,y,b} (\theta )     \ll   \operatorname {exp} (-c \sqrt{\log N}    ).
\end{equation}
Then appeal to \eqref{e:MM} and \eqref{e:MM'}to see the Lemma as written.

The left hand-side of \eqref{e:NNear} equals
\begin{equation*}
 \frac{\phi(y)}{N}\sum_{ \substack{n < N\\ n \equiv b \mod  y }} \Bigl[ \Lambda (n) - \frac{y}{\phi(y)} \Bigr] e (-n \theta ).
\end{equation*}
We will use a trivial bound for $ n \leq \sqrt{N}$.
Apply the Siegel-Walfisz Theorem~\ref{t:SW} and Abel summation   to see that
\begin{align}
& \frac{\phi(y)}{N} \sum_{ \substack{   n < N\\ n \equiv b \mod  y }} \Bigl[ \Lambda (n)  - \frac{y}{\phi(y)} \Bigr]  e (-n \theta )
\\
& \hspace*{1cm} =   \frac{\phi(y)}{N} \bigl(\Psi (N, y, b) - N/ \phi (y) \bigr) e (N \theta )   
\\
& \hspace*{1cm} -    \frac{\phi(y)}{N}(\Psi (\sqrt    N, y, b)  - \sqrt{N}/ \phi (y)) e (\sqrt{N} \theta)
\\ 
& \hspace*{1cm} 
\label{e:SWa}
- 2 \pi i  \frac{\phi(y)}{N} \theta \int _{\sqrt{N}    } ^{N}   \bigl[ \Psi (t, y, b) - t/ \phi (y)) \bigr] e (-\theta t) \; d t + O (\sqrt{N}    ) .
\end{align}
Each term is at most $  \operatorname {exp} (-c_J \sqrt{ \log N}    )$.  The integral is the one that uses the information on $ \theta  $.
We have
\begin{align*}
\eqref{e:SWa}  \ll  \frac{ (\log N) ^{J}}N   \cdot N  \operatorname {exp} (-c_J \sqrt N) .
\end{align*}
This is enough to finish the proof.
\end{proof}

The approximation result on a so-called major arc is below. Recall that from their definition this concerns points in neighborhoods around rationals whose denominators have controlled magnitudes. The statement introduces the parameters $ \ell  = \operatorname {lcm} (y,q)$ and $ g = \gcd  (y,q)$ which play an important role in what follows.
One should also note that the Gauss sum in \eqref{e:Upsilon} depends upon these parameters,
and has itself a complicated expression. Nevertheless, it is  explicitly evaluated in  Lemma~\ref{l:ZERO}.

\begin{lemma}\label{l:Approx}
For all $ J >1$,  there is a $ 0< c < 1$ so that the following holds.
 For $  y,q<\log^{J}(N) $, set
$ \ell = \ell _q = \operatorname {lcm} (y,q)$, and $ g = \gcd (y, q)$.
With  $ (a,q)=1$,  suppose that     $|\xi-\frac{a}{q}|<\frac{\log^{J}(N)}{N}$.
 We have the inequality below.
\begin{align}  \label{e:majorApprox}
\widehat{A}_{N,y,b}( \xi )   & =  { \Upsilon(q,a)}   {\widehat M _{N /\ell }  (\ell( \xi - \tfrac{a}q) ) }
 + O (\operatorname {exp} (-c_J \sqrt{\log N}    )),
\end{align}
where 
\begin{align}
    \label{e:Upsilon}
    \Upsilon( a,q) &= \frac{ \phi (y)} { \phi (\ell ) }
     \sum _{\substack{r\in \mathbb A _q\\ r\equiv b\mod{g}}}
    \operatorname e (-ra/q) .
\end{align}
\end{lemma}

\begin{proof}
The sum  defining  $ \widehat{A}_{N,y,b}(\xi)$ is divided into residue classes mod $ q$.
Consider the conditions
\begin{equation*}
n \equiv b \mod y , \qquad  n \equiv r \mod q .
\end{equation*}
If $g:=\gcd(y,q)$ and  $  b \not\equiv r \mod {g}$, there is no solution.  Otherwise, the conditions above
are equivalent to $ n \equiv \beta_r  \mod \ell $, where $ \ell = \operatorname {lcm} (q,y)$, for some choice of $ \beta _r$.
The choice of  $ \beta _r $   can be made more explicit using a generalized Chinese Remainder Theorem, but that is not necessary for our purposes.

We will write $ \xi = \frac{a}{q} + \theta $, where $ \lvert  \theta \rvert <\frac{\log^{J}(N)}{N}  $.
Observe that
\begin{align*}
\widehat{A}_{N,y,b}(\xi)  & = \frac{\phi(y)}N  \sum _{\substack{r<q \\ r\equiv b\mod{g}}} \; 
 \sum_{\substack{n<N\\n\equiv b\mod{y}\\n\equiv r\mod{q}}} \Lambda (n) \operatorname e (-n  (a/q + \theta  ) )
 \\
 & =
 \frac{\phi(y)} {\phi(\ell) }  \sum _{\substack{r<q\\ r\equiv b\mod{g}}}   \operatorname e (-ra/q)  \cdot  \frac { \phi(\ell )}{N}
 \sum_{\substack{n<N\\n\equiv \beta_r  \mod{ \ell }}}   \Lambda (n) \operatorname e (-n \theta  ) .
\end{align*}
Without loss of generality assume that $n$ is a prime. If $\gcd(r,q)>1$, then $n|q$. It gives the contribution of at most
\begin{equation*}
\frac{\phi(y)q} {N}  \sum_{n|q} \log(n) \ll \frac{\phi(y)q^{\epsilon \textcolor{red}{+1} }} {N}.
\end{equation*}
So we conclude that
\begin{align*}
\widehat{A}_{N,y,b}(\xi)  
    & = \frac{\phi(y)} {\phi(\ell) }  \sum _{\substack{r\in A_q\\ r\equiv b\mod{g}}}   \operatorname e (-ra/q)  \cdot  
        \frac { \phi(\ell )}{N}  \sum_{\substack{n<N\\n\equiv \beta_r  \mod{ \ell }}}   \Lambda (n) \operatorname e (-n \theta  )
    \\
    &  \hspace*{2cm}  + O (\operatorname {exp} (-c_J \sqrt{\log N}    )).
\end{align*}

By our hypotheses, Lemma~\ref{l:nearZero} applies to the inner most sum, for each $ r\in \mathbb A _q$
(with a different choice of $ J$, that is larger by a square).
It follows that
\begin{align} \label{e:betar}
\widehat{A}_{N,y,b}(a/q+ \theta )   =   \frac{\phi(y)} {\phi(\ell)}
\sum_{\substack{r\in \mathbb A _q\\ r\equiv b\mod{g}}}  \operatorname e \big(\frac{ra}{q}\big)    \widehat M _{N /\ell  }  (   \ell \theta )
 + O (\operatorname {exp} (-c_J \sqrt{\log N}    )).
\end{align}
 That completes our proof.

\end{proof}

In \eqref{e:Upsilon}, the sum is a progression restricted Ramanujan's sum
as in Lemma \ref{l:ZERO}.  Applying the latter, we have

\begin{lemma}\label{l:zero}  We have
this equality for $ \Upsilon( a,q) $, defined in \eqref{e:Upsilon}.
\begin{equation}\label{e:zero}
\Upsilon( a,q)  =
\begin{cases}
0   &   \textup{if } 1<g<q \textup{ and }  (g,\frac{q}{g}) >1
\\  \frac{\phi (y) } {\phi (\ell )}
\mu (q/g) \operatorname e ( - abt / g )   &   \textup{if } 1\leq g<q \textup{ and }  (g,\frac{q}{g}) =1
\\
e(-ba/q) & \textup{ if } g=q .
\end{cases}
\end{equation}

\end{lemma}

This formula has implications for how the proof should be organized.  Typically, one expects the Gauss sum at rational $ a/q$ to
decay at a rate dictated by $ q$. That is not the case here.
\begin{enumerate}
\item   If $ q\mid y$, then $ g=q$, and $ \Upsilon( a,q)  = e (-ba/q)$. That is, there is no decay in the height of the Gauss sum.  This is reflection of the fact our sum is restricted to a progression.

\item  If $ 1< g = \gcd (q,y) < q$, and $ (g, q/g) =1$,
there is some decay in the Gauss sum, but only at the rate of $g/q$.

\item   If $ (q,y)=1$, then  $ \lvert  \Upsilon( a,q) \rvert = \phi (q) ^{-1}  $.
These rational points act as if there is no progression.
\end{enumerate}

In particular, there are more than  $ \phi (y) $ rational points $ \frac{a}{q}$ with $ \lvert  \Upsilon( a,q) \rvert \simeq 1 $,
And, our estimates should be independent of $ y$.
This situation is rather different from most of the literature on this type of subject. This next definition is used to keep track of the
relationship between the rational point and the value of the Gauss sum.

\begin{definition}\label{d:height}  Define the \emph{height (with respect to $ y$)} of a rational $ a/q$ with $ (a,q)=1$, or an integer $ q$ to be
\begin{equation}\label{e:height}
h_y (a/q) = h_y (q) =
\begin{cases}
0   &   \textup{if } 1<g<q \textup{ and }  (g,\frac{q}{g}) >1
\\\ell/y
 &   \textup{otherwise}
\end{cases}
\end{equation}
Here, and throughout, $ g = \gcd (y, q)$ and $ \ell = \operatorname {lcm} (y,q)$.    In particular, we have for any $ \epsilon >0$,
\begin{equation}\label{e:htUp}
\lvert  \Upsilon( a,q) \rvert \ll  h_y  (q) ^{-1+ \epsilon },  \qquad \text{whenever } \; h_y (q) >0.
\end{equation}
\end{definition}
We chose to refer to this height as the \emph{Ramanujan height}. The ``traditional'' notion of height, as that term is frequently used in the related literature, is dictated, essentially, by the magnitude of the denominator. For our study, this is not good enough, as it does not take into consideration the restriction to a progression. There is again the dependence on the denominator $q$, which is indicated by the existence of the least common multiple in the formula, however notice that the part of $q$ that actually contributes is the part that is co-prime with $y$. And the same applies to $y$ as well.

\begin{proof}[Proof of \eqref{e:htUp}]  From \eqref{e:Upsilon}, if $ h_y (a/q)=0$, then $ \Upsilon (q,a)$ is also zero.  Otherwise
\begin{align*}
\lvert  \Upsilon( a,q)\rvert  =  \frac{\phi (y)} {\phi ( \ell )}
& = \frac{\phi (y)} {\phi (yq/g)}
\end{align*}
If $ g=q$, the expression above is $ 1$, so that \eqref{e:htUp} trivially holds.
If $ 1\leq g <q$, we have $ \phi (y q/g) \geq \phi (y) \phi (q/g)$, so that \eqref{e:htUp} follows in this case as well.

\end{proof}

It is important to observe that there are a potentially large number of rational points of a given height $ r$.  The exact number is
\begin{align} \label{e:lotsOfRationals}
\sharp \{ a/q \;:\; h_y (q)=r\} &= \sum_{ \substack{q \;:\; \ell/y =r  \\  (g,r)=1}} \phi (q)
\\
&  =  \sum _{ \substack  { g \mid y \\  (g,r)=1  }} \phi(gr)
\\
& =  \phi(r) \sum _{ \substack  { g \mid \frac{y}{ (y,r)}   }} \phi(g) = \phi(r) \frac{y}{ (y,r)}   .
\end{align}
Approaches to different aspects of this question are then limited by these bounds.

This notation is needed for the statement of our principal approximation result.
For $0\leq \xi<1$, let $\ell  : = \textup{lcm} ( y,q) $ and
\begin{align} \label{e:LNyq}
\widehat  L_{N,y}^{a,q} (\xi) 
    & := {\Upsilon(q,a)} {\widehat M _{N /\ell } } \bigl(\ell (\xi- a/q) \bigr) \widehat\eta_{\ell^2} (\xi- a/q),
\end{align}
where $\eta$ is a  non-negative Schwartz function such that $\mathbf{1} _{[-1/16,1/16]} \leq \eta \leq \mathbf{1} _{[-1/4,1/4]}$, and 
\begin{align}
    \label{e:eta}
    \widehat\eta_{t} (\xi) &:= \widehat\eta (  t \xi ).
\end{align}
One should not fail to note that \emph{the cutoff function $ \eta $ above is scaled by $ \ell ^2 $}, to ensure that the major arcs remain disjoint, meaning the support of the multipliers is disjoint as well. The importance of this cutoff will also come into play in the next section when discussing the Low Part, as it allows us to use a multiplicative property of the spatial domain that is important for our estimates.

\begin{thm}\label{t:appthm}
We have the estimate below, uniformly in $ 0\leq \xi < 1$, uniformly for $ b\in \A_y$,
\begin{align}  \label{e:AHiLo}
\widehat{A}_{N,y,b}  &  = \sum_{ q < N ^{-1/10} } \sum_{a\in \mathbb A _q}
\widehat L_{N,y}^{a,q}(\xi)  +  \widehat  E _{N,y}  (\xi ) ,
\\   \label{e:aprox}
\end{align}
where
\begin{align} 
\label{e:Error}
\lVert \widehat   E _{N,y} ( \cdot )\rVert _{L ^{\infty }} & \ll  \operatorname {exp} (-c' \sqrt{\log N}    ).
\end{align}
for a positive constant $c'$ that depends on $y$.
\end{thm}

\begin{proof}
Let $y^8 \simeq R  = 2 ^{r}  \simeq    e^{c\sqrt{\log(N)}}$ for a sufficiently small choice of $ c>0$.
Fix a choice of $ \xi \in \mathbb T $.
Using Dirichlet's Approximation Theorem,  we can choose $0\leq a<q <N^{1/10}$ with $a \in \A_q$, so that  $|\xi-\frac{a}{q}|<\frac{1}{qN^{1/10}}$.
The proof will be organized around the relative sizes of $ q$ and $ R$.

We have this estimate for `large' major arcs.
For each fixed $ s>r/2$, the quantity $\Upsilon(q,a)$ is at most $y2^{-s}$, which has small contribution.
We also bound $\widehat{M_{N/\ell}} \ll 1$.  So, using Lemma \ref{l:ZERO}  we have
\begin{align}
\sum_{s \geq r/2}  \sum_{\frac{a}{q}\in \mathcal R_s}
\widehat L_{N,y}^{a,q}(\xi)&  \ll
\sum_{s \geq r/2} \max _{2 ^{s-1} \leq q < 2 ^{s}}  \left\vert {\Upsilon(q,a)}   {\widehat M _{N /\ell } } (\ell (\xi - \frac{a}{q}) ) \right\vert
\\& \ll  \sum_{s \geq r/2}   \max _{2 ^{s-1} \leq q < 2 ^{s}}  \frac{\phi(y)}{\phi(q)}  \ll R^{\epsilon} \log(N) y R^{-1/2}
\\&\ll R^{-1/4}.     \label{e:Q1}
\end{align}
The implication of this estimate is that we need not concern ourselves with this part of what will end up being the High term of our decomposition.

The remaining analysis is split according to the relative sizes of $ q$ and $ R$.
In the case of $ q \geq R$, concerning the function $ \widehat A _{N,y,b} $ we are in the
setting of classical estimates of Vinogradov.  The particular result we apply to in this setting is the main result of
Balog and Perelli \cite{MR776182}.  It gives us
\begin{align}
\widehat{A}_{N,y,b}(\xi)  & \ll   \frac{y}N  \bigl(NR ^{-1/2}+ \sqrt{RN} + R ^{3/14} N   ^{5/7}\bigr) (\log N  ) ^{18}
\\ & \ll y R^{-1/2}  (\log N  ) ^{18} \ll R^{-1/3} ,
\end{align}
under our assumptions on $ R$ and $ y$.

We establish a corresponding estimate for (the remaining part of)  the High and Low terms.
This will establish \eqref{e:Error}. Assume that $1\leq s<r/2$. We know that  $\ell'(\xi-\frac{a'}{q'})$ should be less than $1/4$ so that $\eta_{q'}(\xi-\frac{a'}{q'})\neq 0$. So $\| \ell' (\xi-\frac{a'}{q'})\| = \ell' (\xi-\frac{a'}{q'})$.
For  $R<q<N^{1/10}$, one must note that
for  $\frac{a}{q}\neq \frac{a'}{q'}\in \mathcal R_s$ we have
\begin{align} \label{e:far}
\left|\xi - \frac{a'}{q'}\right|>\frac{1}{qq'} - \left|\xi - \frac{a}{q}\right| 
    & >  \frac{1}{q'N^{1/10}} -\frac{1}{N^{1/10}q} 
    \\
    & > \frac{1}{N^{1/10}}(2^{-s} - \frac{1}{R}) 
    \\
    & > \frac{1}{N^{1/10}} 2^{-s-1}.
\end{align}
This implies that , for $ \ell' = \textup{lcm} (y ,q')$, we get the upper bound
$$M_{N/\ell'}(\ell'(\xi-\frac{a'}{q'}))\ll \frac{\ell'}{N\| \ell'( \xi - \frac{a'}{q'})\|} \ll \frac{1}{N^{1/2}}.$$

To complete the proof, we now consider the following three cases, dictated by the relative sizes of $q$ and $R$, as well as $s$ and $r$.

\noindent 
\textbf{Case 1:} 
So if $q>R$, let $\ell'=\lcm(q',y)$
\begin{align*}
&\sum_{s\leq r/2}\sum_{\substack{\frac{a}{q}\neq\frac{a'}{q'}\in \mathcal R_s }}
\widehat L_{N,y}^{a',q'}\left(\xi \right)
\\
& \hspace*{1cm} \ll \sum_{s<r/2}\sum_{\substack{\frac{a}{q}\neq \frac{a'}{q'}\in \mathcal R_s }}
\frac{\phi(y)}{  \phi(\ell')}\left\vert\widehat{M}_{N/\ell'}(\ell'(\xi-\frac{a'}{q'})) \right\vert\eta_{q'}\left(\xi-\frac{a'}{q'}\right) 
\\
& \hspace*{1cm} \ll \sum_{s<r/2}\max_{q'\sim 2^s}
\left\vert\widehat{M}_{N/\ell'}(\ell'(\xi-\frac{a'}{q'})) \right\vert .
\end{align*}
There is a uniform bound, in $ s$, on the number of  $ a'/q'$ that contributes above.  So
\begin{align} \label{e:R2}
 \sum_{s< r/2}\sum_{\substack{\frac{a}{q}\neq \frac{a'}{q'}\in \mathcal R_s}}
 \widehat L_{N,y}^{a',q'}\left(\xi\right) \ll \sum_{s<r/2}\frac{1}{N^{1/2}}\simeq \frac{r}{N^{1/2}} \ll  R ^{-1} .
\end{align}
We conclude that both $\widehat{A_N}$ and the High and Low terms  are small if $q>R$. This concludes the proof of \eqref{e:Error} in this case.

\noindent
\textbf{Case 2:} 
If $q<R$ and $s<r/2$ and $a/q\neq a'/q'$, then $ \xi $ and $ a'/q'$ are far apart.  Namely,
\begin{equation*}
\left|\xi - \frac{a'}{q'}\right|>\frac{1}{qq'} - \left|\xi - \frac{a}{q}\right| >  \frac{1}{Rq'} -\frac{1}{N^{1/10}} >  \frac{1}{2Rq'}.
\end{equation*}
This implies that
\begin{align*}
M_{N/\ell' }(\ell'( \xi-\frac{a'}{q'}) )&\ll \frac{ \ell' }{N\|  \ell'(  \xi - \frac{a'}{q'})\|} \ll \frac{2^{s}R}{N}
\end{align*}
Using this, we then have
\begin{align}\label{smallqq'}
& \sum_{s<r/2}\sum_{\substack{\frac{a'}{q'}\in \mathcal R_s\\ \frac{a'}{q'}\neq \frac{a}{q}}}
 \widehat L_{N,y}^{a',q'}\left(\xi\right)  \\
& \hspace*{1cm} \ll \sum_{s<r/2}\sum_{\substack{\frac{a'}{q'}\in \mathcal R_s }}
\frac{\phi(y)}{ \phi(\ell')}\left\vert\widehat{M}_{N/\ell'}(\ell'(\xi-\frac{a'}{q'}))\right\vert \eta_s\left(\xi-\frac{a'}{q'}\right)
\\
& \hspace*{1cm}  \ll \sum_{s<r/2}  \frac{2^{s}R}{N} \ll \frac{R^{2}}{N} < R^{-1}.
\end{align}

\noindent
\textbf{Case 3:} 
 Finally,  when  $q<R $ and $a/q = a'/q'$,    Lemma \ref{l:Approx} immediately implies that 
\begin{align}
\widehat A _{N,y,b}  (\xi )  & =  \widehat L  _{N,y}^{a,q} (\xi ) + O (R ^{-1} ),
\label{e:lastline}
\end{align}
and our proof is complete.
\end{proof}


\section{Properties of the High and Low Parts}  \label{s:HighLow}

We are now ready to define the High and Low part.  
Crucially, the definitions use the Ramanujan height,  
Definition \ref{d:height}. 
For an integer $Q < N ^{1/10}$, with $ Q$ a power of $ 2$, we set
\begin{align}
\label{e:Hi}    {\operatorname {Hi} } _{N,y,Q} & =  \sum_{s \;:\;  Q\leq  h_y (q) \leq N ^{1/10}}
 \sum_{\frac{a}{q}\in \mathcal  R_s}
\  L_{N,y}^{a,q}  ,
\\ \label{e:Lo}
 {\operatorname {Lo} } _{N,y,Q}  &=
\sum_{\substack{  q \;:\;     h_y (q) < Q  }}
 \sum_{\frac{a}{q}\in \mathcal  R_s}
  L_{N,y}^{a,q}  .
\end{align}
The terms $ L _{N,y} ^{a,q}$ are defined in \eqref{e:LNyq}.
Again, the division into High and Low parts is done via the height function $ h _{y} (q)$.
The norm inequalities for these terms are as follows.

\begin{lemma}\label{l:LoHi} 
For all $ \epsilon >0$, $ 1< r < 2$ and $ 1\leq Q <  ( \log N) ^{C r'}$,
and  finite sets $ F\subset  \mathbb Z $ supported in $[0,N]$, we have that there exists $N_{r,y}>0$  so that for $N>N_y,r$, the High term satisfies
\begin{align}\label{e:HighLess1}
\lVert {\operatorname {Hi} } _{N,y,Q} \ast \mathbf 1_{F}\rVert _{\ell ^2 }  & \ll  Q ^{-1+\epsilon }  \lvert  F\rvert ^{1/2} ,
\\ \label{e:HighLess2}
\bigl\lVert \sup _{N = 2 ^{n} > N _{2,y}} \lvert {\operatorname {Hi} } _{N,y,Q} \ast \mathbf 1_{F}\rvert \bigr\rVert _{\ell ^2 }  & \ll  Q ^{-1+\epsilon }
 \lvert  F\rvert ^{1/2} ,
\end{align}
and the Low term satisfies
\begin{align}
 \label{e:LoLess1}
\lVert {\operatorname {Lo} } _{N,y,Q} \ast \mathbf 1_{F}\rVert _{\ell  ^{ \infty } }  & \ll  Q ^{\epsilon }  (y/N) ^{1/r}
\lvert  F\rvert ^{1/r} , 
\\ \label{e:LoLess2}
 \lVert   \sup _{N = 2 ^{n} > N _{r,y}} \lvert {\operatorname {Lo} } _{N,y,Q} \ast \mathbf 1_{F} \rvert \rVert _{\ell ^ {r}  }  & \ll  Q ^{\epsilon }
   \lvert  F\rvert ^{1/r} .
\end{align}
\end{lemma}

The power of $y/N$ is needed to keep the estimate scale free. The constant $N_{r,y}$ is the same as in Theorem \ref{t:maximal}.
The maximal inequalities \eqref{e:HighLess2} and  \eqref{e:LoLess2} are $ \ell ^{r} \to \ell ^{r}$,
so the power of $y/N$ is not needed. (And, they are sharpest when $F$ is restricted to a progression of spacing $y$.)


\subsection{Control of the Low Part}

The estimates for the Low part are more challenging, so we address them first.
Define
\begin{align}  \label{e:Phi}
\widehat  \Phi_{N,q}( \xi) =   {  M _{N /\ell , \ell } } \bigl(\ell  \xi  \bigr)
  \widehat\eta_{\ell^2} (\xi).
\end{align}
We record the elementary inequality for $\Phi_{N,q}$.
\begin{prop}  \label{p:Phi}   We have the estimate
  \begin{equation}
    \Phi_{N,q} (x)  \ll   \eta_ {N} (x).
  \end{equation}
\end{prop}

\begin{proof}
Recall from \eqref{e:eta} that $ \eta $ is a non-negative Schwartz function with
$\mathbf{1} _{[-1/16,1/16]} \leq \eta \leq \mathbf{1} _{[-1/4,1/4]}$.
The function $ \eta _{\ell ^2 }$ then has spatial scale $ \ell  ^2  \leq \sqrt N$, while
$ M _{N /\ell , \ell } $ is an average of length $ N/ \ell $, along a progression of spacing $ \ell  $.
Then, the conclusion above is clear.
\end{proof}

We invert the Fourier transform of the Low term. Experts will recognize that this step is typically routine, leading directly to 
Ramanujan's sums.  In this instance, the proof is notably more complicated.

\begin{lemma} \label{inversefourier}
With the notation of \eqref{e:Lo}, we have
\begin{align}
		  \operatorname{Lo}_{N,y,Q} (x)
			& \leq
			y \, \mathbf 1 _{y|x-b} (x) \sum_{\substack{q' < Q \\  (q', y) =1}} \frac{\lvert \tau _{q'} (x) \rvert  }{  \phi (q' )}     \eta _N  (x) .
			\label{lowinverse}
	\end{align}
Here, $ \ell = \ell _q= \operatorname {lcm} (q, y)$, and  $ \tau _q (x)$ is the Ramanujan function from  \eqref{e:tau}.
\end{lemma}

\begin{proof}
For any $ q$, we can calculate as follows.
	\begin{align*}
		&\F ^{-1} \sum_{a \in \mathbb A _q}  L_{N,y}^{a,q }(\xi-\frac{a}{q})
		\\
			& \hspace*{1cm}  = \sum_{\substack{r\in A_q\\r\equiv b \mod{g}}}e (-ar/q)\sum_{a \in \mathbb A _q}
			\int _{\T} \widehat{M}_{N }(\xi -a/q)  \eta_s(\xi- a/q)  \frac{ \phi(y)  } {  \phi (\ell )}  \, e(x \xi) \, d \xi
			\\
			& \hspace*{1cm}  =     \frac{ \phi(y)  } {  \phi (\ell )} \sum_{\substack{r\in A_q\\r\equiv b \mod{g}}} \sum_{a \in \mathbb A _q}  e (a(x-r)/q)
			\int _{\T} \widehat{M}_{N }(  \theta )  \eta_s( \theta ) e (x \theta )  \; d \theta
			\\
			& \hspace*{1cm}  =  \Phi _{N,q} (x)\frac{ \phi(y)  } { \phi (\ell )}   \sum_{\substack{r\in A_q\\r\equiv b \mod{g}}} \tau _q (x-r)
			\\
			& \hspace*{1cm}  =  \Phi _{N,q} (x)\frac{ \phi(y)  } { \phi (\ell )} \mu (q/g) \mathbf 1_{ (g,\frac{q}{g})=1} \tau_{q/g}(x)\tau_g(x-b).
	\end{align*}

A change of variables allows us to pull  the sum over $ a\in \mathbb A _q$ outside the integral.
And, we use  Lemma \ref{l:cohenprogression} in the last line.

 Take $q' = h_y(q) = \ell/y = q/g $. As $(g,q/g)=1$.
Note that $\Phi_{N,q} $ is just a function of $\ell$. It means that $\Phi_{N,q}$ does not depend on $g$ and only depends on $q' $ and $y $. We have
\begin{align*}
\operatorname{Lo}_{N,y,Q} (x)
			& = \sum_{q: h_y(q) < Q} \Phi _{N,q} (x)\frac{ \phi(y)  } { \phi (\ell )} \mu (q/g) \mathbf 1_{ (g,\frac{q}{g})=1} \tau_{q/g}(x)\tau_g(x-b).
\end{align*}
Observe that $\gcd(q/g,g)=1$, if and only if $\gcd(q/g, y)=1$. This obvious but important property makes the condition $1_{(q/g,g)=1}$ independent of $g$, and only depends on $q',y$. So
\begin{align*}
\operatorname{Lo}_{N,y,Q} (x)
			& = \sum_{q' < Q} \sum_{\substack{g|y\\ (g,q') = 1}} \Phi _{N,q} (x)\frac{ \phi(y)  } { \phi (q'y )} \mu (q')  \tau_{q'}(x)\tau_g(x-b)
			\\
			& = \sum_{\substack{q'<Q\\ (y,q') = 1}} \Phi _{N,q'} (x)\frac{ \phi(y)  } { \phi (q')\phi(y )} \mu (q')  \tau_{q'}(x) \sum_{g|y}\tau_g(x-b).
\end{align*}
Next we use  well-known Ramanujan's sum property
\begin{align*}
\sum_{q|r}\tau_q(n) = \begin{cases}
r & r|n\\
0 & \textup{otherwise.}
\end{cases}
\end{align*}
Applying this property gives us
\begin{align*}
\operatorname{Lo}_{N,y,Q} (x)
			& = \sum_{\substack{q'<Q\\ (y,q') = 1}} \Phi _{N,q} (x)\frac{ \mu (q') } { \phi (q' )}   \tau_{q'}(x) y \mathbf 1_{y|x-b}
			\\
			& \ll
			(y \mathbf 1_{y|x-b})\sum_{\substack{q' < Q \\  (q', y) =1}} \frac{\lvert \tau _{q'} (x) \rvert  }{  \phi (q' )}     \eta _N  (x).
	\end{align*}
Hence we have the result.
\end{proof}

We address the fixed scale estimate \eqref{e:LoLess1} here.
We appeal to details in this proof to prove the maximal estimate \eqref{e:LoLess2}.

\begin{proof}[Proof of \eqref{e:LoLess1}]
We of course use \eqref{lowinverse}, together with  H\"older's inequality.  That gives us
\begin{align}
\operatorname {Lo} _{N,y,Q} * \mathbf{1}_F (x)  & =  \sum _{z} \mathbf{1}_F (z) \operatorname {Lo} _{N,y,Q} (x- z)
\\   \label{e:AB}
& \leq  A_N (x) B_N (x),
\end{align} 
where 
\begin{align} 
A_N(x) ^{r'} & =
y \sum _{x \equiv z+b \mod y}
\
\Biggl[ \sum_{\substack{q' < Q \\  (q', y) =1}} \frac{\lvert \tau _{q'} (x-z) \rvert  }{  \phi (q' )}  \Biggr] ^{r'}\eta _{N} (x-z)
\\
 B_N(x) ^{r}  &=
y \sum _{x \equiv z+b \mod y}  \mathbf{1}_F (z) \eta _N (x-z).
\end{align}
We have treated $y \eta _{N} (x - \cdot)$ as a measure, in our inequality above.
The second term satisfies  $ \lVert B_N\rVert_{ \infty }  \ll   [(y/N) \lvert F \rvert  ] ^{1/r} $.
The first term is controlled by Lemma \ref{l:BourgainRamanujan}.  Recalling the familiar lower bound on
the totient function $ \phi (q ) \gg q ^{1- \epsilon }$, we see that   $ \lVert A_N \rVert_{ \infty  } \ll Q ^{ \epsilon } $. That completes the proof.
\end{proof}

\begin{proof}[Proof of \eqref{e:LoLess2}]
We take advantage of \eqref{e:AB} again, along with the fact that we always have $ \lVert A_N \rVert_{\infty } \ll Q ^{ \epsilon } $ and
\begin{align*}
\bigl\lVert \sup_ {N> N _{y,r}}  B_N   \bigr\rVert _{ \ell ^{s}} & \ll \lvert F \rvert   ^{1/s}, 
    && 1<r<s<2,
\end{align*}
by the usual maximal function estimates.
\end{proof}

\subsection{Properties of the High Term}

The first inequality is the fixed scale $ \ell ^2 $ estimate.
\begin{proof}[Proof of \eqref{e:HighLess1}]
This is entirely elementary. By  Plancherel,  it suffices to estimate
\begin{align*}
\lVert  \widehat  {\operatorname {Hi} } _{N,y,Q}\rVert _{L ^{\infty }}
& \leq  \sum_{s \;:\;   2 ^{s-1} \leq N ^{1/10}}
\Bigl\lVert
\sum_{\substack{  q \;:\;    2 ^{s}\leq q < 2 ^{s+1}  \\  h_y (q) >Q  }}
 \sum_{ a \in \mathbb A _q}
\widehat    L_{N,y}^{a,q}
\Bigr\rVert _{L ^{\infty }}
\\
& \ll  \sum_{s \;:\; 2 ^{s+1} > Q } 2 ^{-s (1- \epsilon )} 
\\&\ll Q ^{-1+ \epsilon }.
\end{align*}
We have taken care to define the functions  $ \{\widehat    L_{N,y}^{a,q}   \;:\;  2 ^{s}   \leq q < 2 ^{s+1}\}$
so that they have disjoint support.  That is done by inserting $ \eta (\ell ^2 (\xi -a/q))$ into the definition
of $ L _{N,y} ^{a,q}$ in \eqref{e:LNyq}.
  And the $ L ^{\infty }$ norm of $ \widehat    L_{N,y}^{a,q}  $ is  at most  $  \phi (y)/ \phi ( \ell ) \leq h _{y} (q) ^{-1+ \epsilon } $.
\end{proof}

For the maximal function estimate \eqref{e:HighLess2}, it is typical to apply the Bourgain Multifrequency Maximal
Inequality from \cite{MR1019960}.  Also, in the typical setting, the height of the rationals and the number of rationals are coupled. In the current
setting, this is no longer true.  Following this path would result in an estimate that is logarithmic in $ y$, because of the
estimate \eqref{e:lotsOfRationals}.

Instead, we recall an inequality from \cite{MR1441675}*{Lemma 2.1}.  It requires the multifrequency base points to share a common denominator,
and  the averages be over scales  large relative to the common denominator.   The constant in the maximal inequality is then
independent of the number of base points.

\begin{lemma}\label{l:notBourgain}
Let $ r_1 ,\dotsc, r_J$ be distinct rational points in $ \mathbb T $, with common denominator $ D < 2 ^{d}$.
Then, we have
\begin{equation}\label{e:notBourgain}
\Bigl\lVert
\sup _{n > 2d }  \Bigl\lvert \mathcal F ^{-1} \Bigl\{
\sum _{j=1} ^{J}   \widehat \eta  (2 ^{n} (\theta - r_j)) \widehat f  (\theta )  \Bigr\} (x)
\Bigr\rvert
\Bigr\rVert _{2} \ll  \lVert f\rVert _{2}.
\end{equation}

\end{lemma}

In our application of this lemma, the number of distinct rational points $a/q$
with $ 2^s < h_y (q/q) \leq 2^{s+1}$ is at most $ Cy^2 2^{2s}$.

\begin{proof}[Proof of \eqref{e:HighLess2}]
In the definition of the High term, we fix $s$ with $2^s >Q/2$, and
consider the maximal function formed over the kernels
\begin{align}
\Gamma _{N,s} =   \sum_{\substack{  q \;:\;    2 ^{s}\leq h_y(q) < 2 ^{s+1} }}
   \sum_{ a \in \mathbb A _q}
  \  L_{N,y}^{a,q} .
\end{align}
The sum above is over at most  $C y^2 2 ^{2s}$ rational points.
A denominator is $g q'$, where $g$ divides $y$ and $q' < 2^{s+1}$.
Their common denominator is then at most $ C y 2^{2s}$.
This means that we can apply  \eqref{e:notBourgain} for the supremum
over $N=2^n > y 2^{2s}$.

Recall that  we only consider $N > N_{y,2} = C y ^{2C}$,  for a large absolute constant $ C$.
For values of $ N_{y,2} \leq  2^n=N < C  y 2^{2s}$, turn to  the fixed scale case, namely \eqref{e:HighLess1},
to conclude that
\begin{equation}
  \big \lVert \sup_{ n \;:\;  N_{y,2} \leq  2^n=N < C  y 2^{2s}}
  \lvert \Gamma_{N,y} \ast f \rvert
  \big \rVert_2  \ll  2^{-s (1- \epsilon)} \lVert f \rVert _2 .
\end{equation}

For the remaining supremum, the definition of $ \Gamma _{N,s}$ needs a slight adjustment in order to
apply \eqref{e:notBourgain}.  Define
\begin{equation*}
\mathcal F \, \tilde \Gamma _{N,s}
=
\sum_{\substack{  q \;:\;    2 ^{s}\leq h_y(q) < 2 ^{s+1} }}
   \sum_{ a \in \mathbb A _q}
   {\Upsilon(q,a)} {\widehat M _{N } }  (\xi- a/q)
\widehat\eta_{ 2 ^{2s}} (\xi- a/q).
\end{equation*}
Here, we have modified the definition of $ L _{N,y} ^{a,q}$ in \eqref{e:LNyq} by
replacing the average ${\widehat M _{N /\ell } } \bigl(\ell (\xi- a/q) \bigr) $ by $ {\widehat M _{N } }  (\xi- a/q) $
and $ \widehat\eta_{\ell^2} (\xi- a/q)$ by $\widehat\eta_{ 2 ^{2s}} (\xi- a/q)$.
With this definition, by a square function argument, we have
\begin{align*}
 \big \lVert \sup_{ n \;:\;  N= 2 ^{n} >  C  y 2^{2s}}
  \lvert (\Gamma_{N,s} - \tilde \Gamma _{N,s} )\ast f \rvert
  \big \rVert_2  ^2
  & \leq
 \sum_{ n \;:\;  N= 2 ^{n} >  C  y 2^{2s}}   \big \lVert
  \lvert (\Gamma_{N,s} - \tilde \Gamma _{N,s} )\ast f \big \rvert
  \rVert_2  ^2
  \\&\ll  2 ^{-2s (1- \epsilon )} \lVert f\rVert_2 ^2 .
\end{align*}
And then, we have a direct application of \eqref{e:notBourgain} to control the supremum below.
\begin{equation*}
\big \lVert \sup_{ n \;:\;  N= 2 ^{n} >  C  y 2^{2s}}
  \lvert  \tilde \Gamma _{N,s} \ast f \rvert
  \big \rVert_2 \ll   2 ^{- s(1- \epsilon )} \lVert f\rVert_2 .
\end{equation*}
We conclude \eqref{e:HighLess2} by summing over $ s $ such that $ 2 ^{s} >Q$.
\end{proof}

\section{Proof of the Main Inequalities} \label{s:proof}

\subsection{The Maximal Function Estimates}

We prove \eqref{e:maximal}. To do so, it suffices to suppose that the function $ f $ on $ \mathbb Z $ is the indicator of
of a set $ F$.   Indeed, we will prove a weak-type estimate for the maximal function.
We need only consider the weak-type estimate at heights $ 0< \lambda  < 1 $.

Fix $ 1< r < 2$, and let $ \epsilon = \frac{r-1}4$.    Below, $ N$ will always be a power of $ 2$.
We trivially have
\begin{align*}
\sum _{N =2 ^{n} < 2 ^{ \lambda ^{-r+1}}} \lVert A _{N,y,b} \mathbf 1_{F}\rVert_1 \ll  \lambda ^{-r+1} \lvert  F\rvert.
\end{align*}
So, we can restrict attention to $ N > 2 ^{ \lambda ^{-r+1}}$.   Importing this condition allows us to take advantage of the maximal inequalities
 \eqref{e:HighLess2} and \eqref{e:LoLess2}, which means we can allow $ Q $ to be as large as
\begin{equation}\label{e:Qbig}
Q  \leq  (\log  2 ^{ \lambda ^{-r+1}}) ^{C r'} = \lambda ^{Cr}.
\end{equation}

Take  $ Q \simeq  \lambda ^{-1+ r/2}$. We will show that for $ N_0 =  \max\{ N _{y,r} , 2 ^{ \lambda ^{-r+1}}\}  $,
\begin{align}  \label{e:Split}
\lvert  \{ \sup_{N > N_0}   A_{N,y,b} \mathbf 1_{F}  > \lambda  \}\rvert
& \ll \{ Q ^{\epsilon } \lambda ^{-r} + Q ^{-2+ \epsilon } \lambda ^{-2} \} \lvert  F\rvert
\\
& \ll \lambda ^{-r+ \epsilon (1-r/2)} \lvert  F\rvert  
\\& \ll \lambda ^{-r+ \epsilon } \lvert  F\rvert .
\end{align}
This proves the restricted weak type estimate $ \ell ^{s,1} \to \ell ^{s, \infty }$, where $ s = r- \epsilon$.
As $ r$ decreases to one, so does $ s$. We deduce the restricted weak type inequality for all $ 1< r < 2$.
Interpolation completes the argument.

Recall our approximation \eqref{e:AHiLo}.  Use the value of $ Q$ above
in the definition of the High and Low terms in \eqref{e:Hi} and \eqref{e:Lo},  respectively.
Then,
\begin{equation*}
A _{N,y} =
{\operatorname {Lo} } _{N,y,Q} +  \operatorname {Hi}  _{N,y,Q} +   E _{N,y,Q}  .
\end{equation*}
Then, by \eqref{e:LoLess2}, we have
\begin{equation*}
\lvert  \{ \sup_{N >  N_0}\lvert   \operatorname {Lo} _{N,y,Q}  \ast  \mathbf 1_{F}\rvert   > \lambda/3  \}\rvert
\ll  Q ^{\epsilon } \lambda ^{-r} \lvert  F\rvert .
\end{equation*}
This is the first half of \eqref{e:Split}.  The estimate below matches the second half of \eqref{e:Split},
and it follows from  \eqref{e:HighLess2}.
\begin{equation*}
\lvert  \{   \sup_{N > N_0} \lvert  \operatorname {Hi} _{N,y,Q}  \ast  \mathbf 1_{F} \rvert  > \lambda /3 \}\rvert
\ll  Q ^{-2 + \epsilon  } \lambda ^{-2} \lvert  F\rvert .
\end{equation*}
Last of all, recalling \eqref{e:Error}, we have
\begin{align*}
\lvert  \{  \sup_{N > N_0} \lvert   \operatorname {E}  _{N,y}  \ast  \mathbf 1_{F}\rvert   > \lambda /3 \}\rvert
& \ll
\lambda ^{-2}
\sum_{N > N_0}
\lVert  \operatorname {E}_{N,y}  \ast  \mathbf 1_{F} \rVert_2 ^2
\\
& \ll  \lambda ^{-2} \operatorname {exp} (- c' \lambda ^{ -(r-1)/2}) \lvert  F\rvert.
\end{align*}
This is better than the second half of \eqref{e:Split}.  So, it completes the proof.

\medskip

\subsection{Fixed Scale Estimates}

We prove the estimate \eqref{fixedscalelpimproving}.
By duality, that estimate is the same as
\begin{equation}\label{e:FSI}
\frac{y}N \langle A _{N,y,b} \mathbf 1_{F},  g\rangle  \ll   \left( \frac{y}{N} \, |F| \right) ^{1/r} \, \left( \frac{y}{N} \lvert G \rvert   \right)^{1/r} ,
\end{equation}
where $ F$ and $G$ are subsets of an interval $ E$ of length $ N$.

Observe that trivially
\begin{align*}
\frac{y}N \langle A _{N,y,b} \mathbf 1_{F},  g\rangle
\ll  \log N   \, \frac{y}{N} \, |F|  \cdot \frac{y}{N} \lvert G \rvert  .
\end{align*}
This implies that the inequality \eqref{e:FSI} is true,  unless
\begin{equation}\label{e:SMALL}
\frac{ y ^{2} } {N ^2   }  \lvert  F\rvert \cdot  \lvert G \rvert     \ll   (\log N) ^{-r'}.
\end{equation}
So it suffices to only study this case.

We take $ N > N _{r,y}$, and $ 0< \epsilon < \frac{r-1}{100}$ small, and apply the High/Low decomposition with parameter $ Q$ to be determined later.
Using the estimates  \eqref{e:HighLess1} and \eqref{e:LoLess1}, we have
\begin{align*}
\frac{y}N \langle   \operatorname{Hi}_{N,y,Q}  \ast  \mathbf 1_{F},  g \rangle
& \ll  Q^{-1+\epsilon} \, \left( \frac{y}{N} \, |F| \right)^{1/2} \, \left( \frac{y}{N} \lvert G \rvert   \right)^{1/2} ,
\\
\frac{y}N (  \operatorname{Lo}_{N,y,Q}  \ast  \mathbf 1_{F},  g)
& \ll  Q^\epsilon \, \left( \frac{y}{N} \, |F| \right)^{1/r} \,  \frac{y}{N}  \lvert G \rvert  .
\end{align*}
Optimize over $ Q$ so that the right hand sides above are approximately equal. We obtain
\begin{equation*}
Q^{1+2\epsilon} \simeq  \left( \frac{y}{N} \, |F| \right)^{\frac{1}{2} - \frac{1}{r}} \, \left( \frac{y}{N}  \lvert G \rvert   \right)^{-\frac{1}{2}}.
\end{equation*}
By \eqref{e:SMALL}, this is an allowed choice for us.

So our estimate becomes
\begin{equation*}
	\frac{y}N \langle A _{N,y,b} \mathbf 1_{F},  g\rangle  \ll
	\left( \frac{y}{N} \, |F| \right)^{ \frac{1}{r} + \epsilon'}
	 \left( \frac{y}{N} \lvert G \rvert   \right)^{ \frac{1}{r} + \epsilon ' }.
\end{equation*}
Above $ \epsilon' < c _{r} \epsilon $.  Thus, we see that \eqref{e:FSI} holds for $1 < r < 2$.

\begin{remark}
The estimate above could be improved to a sparse bound for the maximal function. However, the notion of a sparse bound would have to be refined to one that is adapted to progressions.  Not having a
ready application of such a result, we do not pursue the details herein.
\end{remark}


\end{document}